\documentclass[12pt]{amsart}

\usepackage{amssymb,latexsym}
\usepackage{enumerate}

\makeatletter
\@namedef{subjclassname@2010}{%
  \textup{2010} Mathematics Subject Classification}
\makeatother

\newtheorem{thm}{Theorem}[section]
\newtheorem{prop}[thm]{Proposition}
\newtheorem{cor}[thm]{Corollary}
\newtheorem{lem}[thm]{Lemma}

\theoremstyle{definition}

\newtheorem{rem}[thm]{Remark}

\newcommand{\Sph}{\mathbb S}
\newcommand{\R}{\mathbb{R}}

\DeclareMathOperator{\dist}{dist}

\begin{document}

\title
{Fractional Hardy-Sobolev-Maz'ya inequality for domains}

\author[B.~Dyda]{Bart{\l}omiej~Dyda}
\address{
Bart{\l}omiej~Dyda\\
Faculty of Mathematics\\ University of Bielefeld\\
Postfach 10 01 31,
D-33501 Bielefeld, Germany
\and
 Institute of Mathematics and Computer Science\\ Wroc{\l}aw University of Technology\\
Wybrze\.ze Wyspia\'nskiego 27,
50-370 Wroc{\l}aw, Poland}

\author[R.~L.~Frank]{Rupert~L.~Frank}
\address{Rupert~L.~Frank\\
Department of Mathematics,
Princeton University, Princeton, NJ 08544, USA}

\subjclass[2010]{Primary 26D10; Secondary 46E35, 31C25}

\keywords{fractional Hardy-Sobolev-Maz'ya inequality,
  fractional Hardy inequality}

\sloppy \footnotetext{
Work supported by the DFG through SFB-701 'Spectral Structures and Topological Methods in Mathematics'
and by grant N N201 397137, MNiSW (B.D.) and by U.S. NSF grant PHY–1068285 (R.L.F.)
}

\begin{abstract}
We prove a fractional version of the Hardy--Sobolev--Maz'ya inequality for arbitrary domains and $L^p$ norms with $p\geq 2$. This inequality combines the fractional Sobolev and the fractional Hardy inequality into a single inequality, while keeping the sharp constant in the Hardy inequality.
\end{abstract}

\maketitle

\section{Introduction}
We are concerned here with the fractional Hardy inequality in an arbitrary
domain $\Omega\subsetneq \R^N$, which states that if $1<p<\infty$ and $0<s<1$ with $ps>1$, then
\begin{align}\label{eq:hardy}
\iint_{\Omega\times\Omega} \frac{|u(x)-u(y)|^p}{|x-y|^{N+ps} } \,dx\,dy
\geq \mathcal D_{N,p,s} \int_{\Omega} \frac{|u(x)|^p}{m_{ps}(x)^{ps}}\,dx
\end{align}
for all $u\in$ \textit{\r{W}}$^s_p(\Omega)$, the closure of $C_c^\infty(\Omega)$ with respect to the left side of \eqref{eq:hardy}.
The pseudodistance $m_{ps}(x)$ is defined in (\ref{eq:malpha}); its most important property for the present discussion is that for \emph{convex} domains $\Omega$ we have $m_{ps}(x) \leq \dist(x, \Omega^c)$. We denote by $\mathcal D_{N,p,s}$ the \emph{sharp} constant in~\eqref{eq:hardy}, which was recently found by Loss and Sloane \cite{LossSloane} and is explicitly given in \eqref{eq:hardyconst} below. This constant is independent of $\Omega$ and coincides with that on the halfspace which was earlier found in \cite{KBBD-bc,FrankSeiringer}.

By the (well-known) Sobolev inequality the left side of \eqref{eq:hardy} dominates an $L_q$-norm of $u$. Our main result, the fractional HSM inequality, states that the left side of \eqref{eq:hardy}, even after subtracting the right side, is still strong enough to dominate this $L_q$-norm. More precisely, we shall prove 

\begin{thm}\label{thm:HSM}
 Let $N\geq 2$, $2\leq p<\infty$ and $0<s<1$ with $1<ps<N$. Then there is a constant $\sigma_{N,p,s}>0$ such that
\begin{align}\label{eq:main}
\iint_{\Omega\times\Omega} \frac{|u(x)-u(y)|^p}{|x-y|^{N+ps} } \,dx\,dy
 - \mathcal D_{N,p,s} \int_{\Omega} \frac{|u(x)|^p}{m_{ps}(x)^{ps}}\,dx 
\geq \sigma_{N,p,s} \, \left( \int_{\Omega} |u(x)|^q \,dx \right)^{p/q}
\end{align}
for all open $\Omega\subsetneq\R^N$ and all $u\in$ \textit{\r{W}\,}$^s_p(\Omega)$, where $q=Np/(N-ps)$.
\end{thm}

Inequality \eqref{eq:main} has been conjectured in \cite{FrankSeiringer} in analogy to the local HSM inequalities \cite{M,BFT}.
Recently, Sloane \cite{Sloane} found a remarkable proof of \eqref{eq:main} for $p=2$ and $\Omega$ being a half-space. Our result generalizes this to any $p\geq 2$ and any $\Omega$. We emphasize that our constant $\sigma_{N,p,s}$ can be chosen independently of $\Omega$. Therefore Theorem \ref{thm:HSM} is the fractional analog of the main inequality of \cite{FLHSM}, which treats the local case.

We know explain the notation in \eqref{eq:main}. The sharp constant \cite{LossSloane} in \eqref{eq:hardy} is
\begin{equation}\label{eq:hardyconst}
\mathcal D_{N,p,s} = 2 \pi^{\frac{N-1}2} \frac{\Gamma(\frac{1+ps}2)}{\Gamma(\frac{N+ps}2)}
\int_0^1 \left( 1 - r^{(ps-1)/p} \right)^p \frac{dr}{(1-r)^{1+ps}} \,.
\end{equation}
In the special case $p=2$ we have
\[
 \mathcal D_{N,2,s} = 
2\pi^{\frac{N-1}{2}} \frac{\Gamma(\frac{1+2s}{2})}{\Gamma(\frac{N+2s}{2})}
\frac{B\left(\frac{1+2s}{2}, 1-s \right)
  -2^{2s}}{ 2^{2s+1}s}  = 2\kappa_{N,2s},
\]
where $\kappa_{N,2s}$ is the notation used in \cite{KBBD-bc, LossSloane, DydaHint}. We denote
\begin{align}
d_{\omega}(x) &= \inf\{ |t| : x+t\omega \not\in \Omega \}, \quad x\in \R^N,\, \omega \in \Sph^{N-1},
\end{align}
where $\Sph^{N-1}=\{x\in \R^N: |x|=1\}$ is the $(N-1)$-dimensional unit sphere. Following \cite{LossSloane} we set for $\alpha>0$
\begin{align}\label{eq:malpha}
 m_\alpha(x)
&= \left( \frac{2\pi^{\frac{N-1}{2}} \Gamma(\frac{1+\alpha}{2}) }{\Gamma(\frac{N+\alpha}{2}) } \right)^{\frac{1}{\alpha}}
 \left(\int_{\Sph^{N-1}} \frac{d\omega}{d_\omega(x)^\alpha} \right)^{-\frac{1}{\alpha}}, 
\end{align}
which is analogous to the pseudodistance $m(x)$ of Davies \cite[Theorem~5.3.5]{Davies}.
We recall that for convex domains $\Omega$, we have $m_\alpha(x) \leq d(x)$,
see \cite{LossSloane}.

This paper is organized as follows. In the next three sections we present three independent proofs of \eqref{eq:main}, but only the last one in full generality. In Section~\ref{sec:proof2}, we use the ground state representation for half-spaces as the starting point. This allows us to obtain \eqref{eq:main} for half-spaces and any $p\geq 2$. In Section~\ref{sec:proof1} we derive a~fractional Hardy inequality (\ref{Hball}) for balls  with two additional terms, and then deduce \eqref{eq:main} in case when $p=2$ and $\Omega$ is a ball or a half-space. In the last section, we extend the method developed in \cite{FLHSM} and use results from \cite{GRR} and \cite{LossSloane} to prove Theorem~\ref{thm:HSM} for arbitrary domains.

\subsection*{Acknowledgment}
The authors would like to thank M. Loss and C. Sloane for useful discussions.


\section{The inequality on a halfspace}\label{sec:proof2}

In this section, we prove Theorem~\ref{thm:HSM} in the particular case when $\Omega=\R^N_+=\{x\in\R^N:\ x_N>0\}$. Our starting point is the inequality
\begin{equation}
\label{eq:gsr}
\iint_{\R^N_+\times\R^N_+} \frac{|u(x)-u(y)|^p}{|x-y|^{N+ps} } \,dx\,dy
 - \mathcal D_{N,p,s} \int_{\R^N_+} \frac{|u(x)|^p}{x_N^{ps}}\,dx 
\geq c_p J[v] \,,
\end{equation}
where $c_p$ is an explicit, positive constant (for $p=2$ this is an identity with $c_2=1$),
$$
J[v] := \iint_{\R^N_+\times\R^N_+} \frac{|v(x)-v(y)|^p}{|x-y|^{N+ps} } (x_N y_N)^{(ps-1)/2} \,dx\,dy \,,
$$
and $v(x):=x_N^{-(ps-1)/p} u(x)$. This inequality was derived in \cite{FrankSeiringer}, using the `ground state representation' method from \cite{FrSe1}. We note that $m_{ps}(x)=x_N$ in the case of a halfspace, as a quick computation shows (see also \cite[(7)]{LossSloane}).

In order to derive a lower bound on $J[v]$ we make use of the bound
\begin{equation*}
(x_N y_N)^a \geq \min\{x_N^{2a},y_N^{2a}\} = 2a \int_0^\infty \chi_{(t,\infty)}(x_N) \chi_{(t,\infty)}(y_N) t^{2a-1} \,dt
\end{equation*}
for $a>0$. Combining this inequality with the fractional Sobolev inequality (see Lemma \ref{sob} below) and Minkowski's inequality, we can bound
\begin{align*}
 J[v] & \geq (ps-1) \int_0^\infty \iint_{\{x_N>t,\,y_N>t\}} \frac{|v(x)-v(y)|^p}{|x-y|^{N+ps} } \,dx\,dy \ t^{ps-2}\,dt \\
& \geq (ps-1) \mathcal C_{N,p,s} \int_0^\infty \left( \int_{\{x_N>t\}} |v(x)|^q \,dx \right)^{p/q} \,t^{ps-2}\,dt \\
& \geq (ps-1) \mathcal C_{N,p,s} \left( \int_{\R^N_+} |v(x)|^q \left( \int_0^{x_N} t^{ps-2}\,dt \right)^{q/p} dx \right)^{p/q} \\
& = \mathcal C_{N,p,s} \left( \int_{\R^N_+} |v(x)|^q \, x_N^{q(ps-1)/p} \,dx \right)^{p/q} \,.
\end{align*}
Recalling the relation between $u$ and $v$ we arrive at \eqref{eq:main}. This completes the proof of Theorem \ref{thm:HSM} when $\Omega=\R^N_+$.
\qed

\medskip

In the previous proof we used the Sobolev inequality on half-spaces for functions which do not necessarily vanish on the boundary. For the sake of completeness we include a short derivation of this inequality. The precise statement involves the closure $\dot W^s_p(\R^N_+)$ of $C_c^\infty(\overline{\R^N_+})$ with respect to the left side of \eqref{eq:hardy}.

\begin{lem}\label{sob}
 Let $N\geq 1$, $1\leq p<\infty$ and $0<s<1$ with $ps<N$. Then there is a constant $\mathcal C_{N,p,s}>0$ such that
\begin{align*}\label{eq:sob}
\iint_{\R^N_+\times\R^N_+} \frac{|u(x)-u(y)|^p}{|x-y|^{N+ps} } \,dx\,dy
\geq \mathcal C_{N,p,s} \, \left( \int_{\R^N_+} |u(x)|^q \,dx \right)^{p/q}
\end{align*}
for all $u\in\dot W^s_p(\R^N_+)$, where $q=Np/(N-ps)$.
\end{lem}

\begin{proof}
 If $\tilde u$ denotes the even extension of $u$ to $\R^N$, then
\begin{align*}
 \iint_{\R^N\times\R^N} \frac{|\tilde u(x)-\tilde u(y)|^p}{|x-y|^{N+ps} } \,dx\,dy
&=  \, 2 \iint_{\R^N_+\times\R^N_+} \frac{|u(x)-u(y)|^p}{|x-y|^{N+ps} } \,dx\,dy \\
&\!\!\!\!\!\!\!\!\!\!\!\!\!\!\!\!\!\!\!\!\!\!\!\!\!\!\!\!\!\!\!\!\!\!\!\!\!\!\!\!\!\!\!\!\!\!
 + 2 \iint_{\R^N_+\times\R^N_+} \frac{|u(x)-u(y)|^p}{(|x'-y'|^2+(x_N+y_N)^2)^{(N+ps)/2}} \,dx\,dy \\
& \leq \, 4 \iint_{\R^N_+\times\R^N_+} \frac{|u(x)-u(y)|^p}{|x-y|^{N+ps} } \,dx\,dy \,.
\end{align*}
On the other hand, by the `standard' fractional Sobolev inequality on $\R^N$ (see, e.g., \cite{FrSe1} for explicit constants) the left side is an upper bound on
\begin{equation*}
 \mathcal S_{N,p,s} \, \left( \int_{\R^N} |\tilde u(x)|^q \,dx \right)^{p/q}
= 2^{p/q} \mathcal S_{N,p,s} \, \left( \int_{\R^N_+} |u(x)|^q \,dx \right)^{p/q} \,. \qedhere
\end{equation*}
\end{proof}

\begin{rem}
 The above proof of the fractional HSM inequality works analogously in the local case, that is, to show that
\begin{equation}\label{eq:hsm}
\int_{\R^N_+} |\nabla u|^p \,dx - \left(\frac{p-1}{p}\right)^p \int_{\R^N_+} \frac{|u|^p}{x_N^{p}}\,dx
\geq \sigma_{N,p,1} \left( \int_{\R^N_+} |u|^q \,dx \right)^{p/q},
\ q= \frac{Np}{N-p},
\end{equation}
for $u\in$ \textit{\r{W}\,}$^1_p(\R^N_+)$ when $N\geq 3$ and $2\leq p<N$. Again, the starting point \cite{FrSe1} is to bound the left side from below by an explicit constant $c_p>0$ times
$$
\int_{\R^N_+} |\nabla v|^p x_N^{p-1} \,dx \,,
\qquad v= x_N^{-(p-1)/p} u \,.
$$
(For $p=2$, this is an identity with $c_2=1$.) Next, we write $x_N^a = a \int_0^\infty \chi_{(t,\infty)}(x_N) t^{a-1} \,dt$ and use Sobolev's inequality on the half-space $\{x_N>t\}$ together with Minkowski's inequality. Note that the sharp constants in this half-space inequality are known explicitly (namely, given in terms of the whole-space constants via the reflection method of Lemma \ref{sob}).
\end{rem}

The sharp constant in \eqref{eq:hsm} for $p=2$ and $N=3$ was found in \cite{BFL}. We think it would be interesting to investigate this question for the non-local inequality \eqref{eq:main} and we believe that \cite{Sloane} is a promising step in this direction.


\section{The inequality on a ball}\label{sec:proof1}

Our goal in this section is to prove a fractional Hardy--Sobolev--Mazya inequality on the ball $B_r\subset\R^N$, $N\geq 2$, of radius $r$ centered at the origin. The argument follows that from the previous section, but is more involved. More precisely, we shall prove

\begin{prop}\label{ballprop}
Let $N\geq 2$, $p=2$ and $\frac{1}{2}<s<1$. Then there is a constant $c=c(s,N)>0$ such that for every $0<r<\infty$ and $u\in$\textit{\r{W}\,}$^s_2(B_r)$,
\begin{align}
\int_{B_r} \! \int_{B_r}
\frac{|u(x)-u(y)|^2}{|x-y|^{N+2s}} \,dx\,dy & - \mathcal D_{N,p,s} \int_{B_r} \frac{(2r)^{2s}}{(r^2-|x|^2)^{2s}} |u(x)|^2 \,dx
\nonumber\\
& \geq c \left(\int_{B_r} {|u(x)|^{q}}\,dx \right) ^ {2/q}, \label{HSMball}
\end{align}
where $q=2N/(N-2s)$.
\end{prop}

This proves Theorem \ref{thm:HSM} in the special case $\Omega=B_r$ and $p=2$ with $m_{2s}(x)$ replaced by $(r^2-|x|^2)/2r$. We note that $(r^2-|x|^2)/2r \leq \dist (x,B_r^c)$ for $x\in B_r$. (As an aside we note, however, that it is \emph{not always} true that $(r^2-|x|^2)/2r$ is greater than $m_{2s}(x)$. Indeed, take $x=0$ and $N=2$.) 

We also note that Proposition \ref{ballprop} implies Theorem \ref{thm:HSM} for $\Omega=\R^N_+$ (and $p=2$). Indeed, by translation invariance the proposition implies the inequality also on balls $B(a_r, r)$ centered at $a_r=(0,\ldots,0,r)$. We have $\dist(x,B(a_r, r)^c) \leq \dist(x, (\R^N_+)^c)$, and hence the result follows by taking $r \to\infty$.

The crucial ingredient in our proof of Proposition \ref{ballprop} is

\begin{lem}\label{Hball}
Let $N\geq 2$, $\frac{1}{2}<s<1$ and define $w_N(x)=(1-|x|^2)^{\frac{2s-1}{2}}$ for $x\in B_1\subset\R^N$. Then for all $u\in$\textit{\r{W}\,}$^s_2(B_1)$
\begin{align}
\int_{B_1} \! \int_{B_1}
\frac{|u(x)-u(y)|^2}{|x-y|^{N+2s}} \,dx\,dy & - \mathcal D_{N,2,s} \int_{B_1} \frac{2^{2s}}{(1-|x|^2)^{2s}} |u(x)|^2 \,dx
\nonumber\\
& \geq \tilde J[v] + c \int_{B_1} |v(x)|^2 \,dx \,,
 \label{hardyinball}
\end{align}
where $v= u/w_N$,
$$
\tilde J[v] = \int_{B_1} \!\int_{B_1} |v(x)-v(y)|^2 \frac{w_N(x)w_N(y)}{|x-y|^{N+2s}} \,dx\,dy
$$
and $c= s^{-1}(2^{2s-1}-1)|\Sph^{N-1}|>0$.
\end{lem}

This inequality is somewhat analogous to \eqref{eq:gsr} in the previous proof. We emphasize, however, that there are two terms on the right side of \eqref{hardyinball} and we will need both of them. Accepting this lemma for the moment, we now complete the

\begin{proof}[Proof of Proposition \ref{ballprop}]
By scaling, we may and do assume that $r=1$, that is, we consider only the unit ball $B_1\subset \R^N$. We put $v=u/w_N$ with $w_N$ defined in Lemma \ref{Hball}. According to that lemma, the left side of \eqref{HSMball} is bound from below by
\begin{align}
\tilde J[v] + c \int_{B_1} |v(x)|^2 w_N(x)^2 \,dx, \label{eq:reduction}
\end{align}
(Here we also used that $w_N\leq 1$.)
For $x,y\in B_1$ we have
\begin{align*}
 w_N(x)w_N(y) &\geq \min\{(1-|x|^2)^{2s-1}, (1-|y|^2)^{2s-1}\}  \nonumber\\
&= (2s-1) \int_0^1 \chi_{(t,1]}(1-|x|^2) \chi_{(t,1]}(1-|y|^2) t^{2s-2}\,dt \,, 
\end{align*}
and therefore,
\begin{align*}
& \tilde J[v] + c \int_{B_1} |v(x)|^2 w_N(x)^2\,dx \nonumber\\
& \geq (2s-1)
  \int_0^1 \left( \int_{B_{\sqrt{1-t}}} \! \int_{B_{\sqrt{1-t}}}  \frac{|v(x)-v(y)|^2 }{|x-y|^{N+2s}} \,dx\,dy
+ c \int_{B_{\sqrt{1-t}}} |v(x)|^2 \,dx \right) t^{2s-2}\,dt \,.
\end{align*}
The fractional Sobolev inequality \cite[(2.3)]{ChenKumagai} and a scaling argument imply that there is a $\tilde c>0$ such that for all $r>0$,
\begin{equation*}
r^{2s} \int_{B_r} \! \int_{B_r}
   \frac{|v(x)-v(y)|^2 }{|x-y|^{N+2s}} \,dx\,dy
+ c \int_{B_r} |v(x)|^2 \,dx  \geq \tilde c r^{2s} \left( \int_{B_r} |v(x)|^{q}\,dx \right)^{2/q}.
\end{equation*}
Combining the last two relations and applying Minkowski's inequality, we may bound
\begin{align}
& \tilde J[v] + c \int_{B_1} |v(x)|^2 w_N(x)^2\,dx \nonumber\\
& \geq (2s-1) \tilde c \int_0^1 \left(  \int_{B_{\sqrt{1-t}}} |v(x)|^{q} \,dx \right)^{\frac{2}{q}} (\sqrt{1-t})^{2s} t^{2s-2}\,dt  \nonumber\\
&\geq  (2s-1) \tilde c
\left( \int_{B_1} |v(x)|^{q} \left( \int_0^{1-|x|^2} (1-t)^s t^{2s-2}\,dt \right)^{\frac{q}{2}} dx \right)^{\frac{2}{q}}. \label{suma}
\end{align}
We observe that
\[
 \int_0^{1-|x|^2} (1-t)^s t^{2s-2}\,dt \geq B(s+1,2s-1) (1-|x|^2)^{2s-1},
\]
which follows from the fact that $y\mapsto \int_0^y (1-t)^s t^{2s-2}\,dt / \int_0^y t^{2s-2}\,dt$
is decreasing on $(0,1)$.
This allows us to bound the expression in (\ref{suma}) from below by
\begin{align*}
(2s-1)& B(s+1,2s-1) \tilde c \left( \int_{B_1} |v(x)|^{q} (1-|x|^2)^{(s-1/2)q} dx \right)^{\frac{2}{q}}\\
 &=
(2s-1) B(s+1,2s-1) \tilde c \left( \int_{B_1} |u(x)|^{q}  dx \right)^{\frac{2}{q}} \,,
\end{align*}
and we are done.
\end{proof}

This leaves us with proving Lemma \ref{Hball}. We need to introduce some notation. The regional Laplacian (see, e.g., \cite{MR2214908})
on an open set $\Omega\subset\R^N$ is, up to a multiplicative constant, given by
\[
 L_\Omega u(x) =  \lim_{\varepsilon\to 0^+}
\int_{\Omega \cap \{|y-x|>\varepsilon\}}\frac{u(y)-u(x)}{|x-y|^{N+2s}} \,dy.
\]
This operator appears naturally in our context since
$$
\int_\Omega \overline{u(x)} (L_\Omega u)(x) \,dx 
= -\frac12 \int_\Omega \! \int_\Omega \frac{|u(x)-u(y)|^2}{|x-y|^{N+2s}} \,dx\,dy \,.
$$
Our proof of Lemma \ref{Hball} relies on a pointwise estimate for $L_{B_1} w_N$. In dimension $N=1$ this can be computed explicitly and we recall from \cite[Lemma 2.1]{DydaHint} that
\[
 - L_{(-1,1)}w_1(x) = \frac{(1-x^2)^{\frac{-2s-1}{2}}}{2s}  \left(
B( s+{\textstyle \frac{1}{2}} ,1-s) - (1 - x)^{2s} + (1 + x)^{2s} \right) \,.
\]
Hence, by \cite[(2.3)]{DydaHint},
\begin{equation}\label{Lw1}
 - L_{(-1,1)}w_1(x) \geq c_1 (1-x^2)^{\frac{-2s-1}{2}} + c_2 (1-x^2)^{\frac{-2s+1}{2}},
\end{equation}
where
\[
 c_1=\frac{B( s+{\textstyle \frac{1}{2}} ,1-s) - 2^{2s}}{2s},\quad
c_2=\frac{2^{2s}-2}{2s} \,.
\]

\begin{lem}\label{laplasjanupball}
Let $N\geq 2$ and let $w_N$ be as in Lemma \ref{Hball}. Then
\[ 
 -L_{B_1} w_N(x) \geq  \frac{c_1}{2} \int_{\Sph^{N-1}}|h_N|^{2s} dh
\cdot  (1-|x|^2)^{-\frac{2s+1}{2}}
+ \frac{c_2}2|\Sph^{N-1}| \cdot (1-|x|^2)^{-\frac{2s-1}{2}} \,.
\]
\end{lem}

\begin{proof}
By rotation invariance we may assume that $\mathbf{x}=(0,0,\ldots,0,x)$. With the notation $p=\frac{2s-1}{2}$ we have
\begin{align*}
 -L_{B_1}w_N(\mathbf{x}) &= p.v. \int_{B_1} \frac{(1-|\mathbf{x}|^2)^p - (1-|y|^2)^p}
     {|\mathbf{x}-y|^{N+2s}}\,dy\\
&=\frac{1}{2} \int_{\Sph^{N-1}} dh
\; p.v.\int_{-xh_N  - \sqrt{x^2h_N^2-x^2+1}}^{-xh_N  + \sqrt{x^2h_N^2-x^2+1}}
\frac{ (1-|x|^2)^p - (1-|x+h t|^2)^p}{|t|^{1+2s}} \,dt \,.
\end{align*}
We calculate the inner principle value integral by changing the variable
$t=-xh_N + u \sqrt{x^2h_N^2-x^2+1}$
\begin{align*}
g(x,h)&:=
p.v.\int_{-xh_N  - \sqrt{x^2h_N^2-x^2+1}}^{-xh_N  + \sqrt{x^2h_N^2-x^2+1}}
\frac{ (1-|x|^2)^p - (1-|x+h t|^2)^p}{|t|^{1+2s}} \,dt\\
&=p.v. \int_{-1}^1
  \frac{(1-x^2)^p - (1-u^2)^p(1-x^2+x^2h_N^2)^p}
    {|-xh_N + u \sqrt{x^2h_N^2-x^2+1}|^{1+2s}}\,
 \sqrt{x^2h_N^2-x^2+1}\,du\\
&=
(1-x^2+x^2h_N^2)^{p-s}
 p.v.\int_{-1}^1 \frac{ (1- \frac{x^2h_N^2}{1-x^2+x^2h_N^2})^p - (1-u^2)^p}
 { |u - \frac{xh_N}{\sqrt{1-x^2+x^2h_N^2}}|^{1+2s}}\,du\\
&=(1-x^2+x^2h_N^2)^{-1/2}
(-L_{(-1,1)}w_1)(\frac{xh_N}{\sqrt{1-x^2+x^2h_N^2}}) \,.
\end{align*}
Hence by (\ref{Lw1}) we have
\begin{align*}
g(x,h) &\geq
(1-x^2+x^2h_N^2)^{-1/2} \bigg(
 c_1 (1- \frac{x^2h_N^2}{1-x^2+x^2h_N^2})^{\frac{2s-1}{2}-2s}\\
&\qquad\qquad\qquad\qquad\qquad\quad
  + c_2 (1- \frac{x^2h_N^2}{1-x^2+x^2h_N^2})^{\frac{2s-1}{2}-2s+1}
\bigg)\\
&= c_1(1-x^2+x^2h_N^2)^{s}(1-x^2)^{-\frac{2s+1}{2}}+
 c_2(1-x^2+x^2h_N^2)^{s-1}(1-x^2)^{-\frac{2s-1}{2}}\\
&\geq c_1 |h_N|^{2s} (1-x^2)^{-\frac{2s+1}{2}}
+ c_2 (1-x^2)^{-\frac{2s-1}{2}} \,.
\end{align*}
Thus
\begin{align*}
 -L_{B_1}w_N(\mathbf{x}) &=
 \frac{1}{2} \int_{\Sph^{N-1}}g(x,h) dh \\
&\geq
 \frac{c_1}{2} \int_{\Sph^{N-1}}|h_N|^{2s} dh
\cdot  (1-x^2)^{-\frac{2s+1}{2}}
+ \frac{c_2}2|\Sph^{N-1}| \cdot (1-x^2)^{-\frac{2s-1}{2}} \,,
\end{align*}
and we are done.
\end{proof}

Finally, we are in position to give the

\begin{proof}[Proof of Lemma \ref{Hball}]
We use the ground state representation formula \cite{FrSe1}, see also \cite[Lemma 2.2]{DydaHint},
\begin{align*}
\int_{B_1} \! \int_{B_1}
\frac{|u(x)-u(y)|^2}{|x-y|^{N+2s}} \,dx\,dy + 2\int_{B_1} \frac{Lw_N(x)}{w_N(x)} |u(x)|^2 \,dx
= \tilde J[v]
\end{align*}
with $u=w_N v$ and $\tilde J$ as defined in the lemma. The assertion now follows from Lemma \ref{laplasjanupball}, which implies that
$$
-2 \frac{Lw_N(x)}{w_N(x)} \geq \mathcal D_{N,2,s} \frac{2^{2s}}{(1-|x|^2)^{2s}} + c (1-|x|^2)^{-2s+1}
$$
with $c=c_2 |\Sph^{N-1}|>0$. Indeed, here we used $2^{2s-1} \mathcal D_{1,2,s} = c_1$ and
\[
 \mathcal D_{N,2,s} = \mathcal D_{1,2,s}\cdot \frac{1}{2}\int_{\Sph^{N-1}} |h_N|^{2s}\,dh \,.
\]
as a quick computation shows.
\end{proof}


\section{The inequality in the general case}\label{sec:proof3}

In this section we shall give a complete proof of Theorem \ref{thm:HSM}. Our strategy is somewhat reminiscent of the proof of the Hardy--Sobolev--Maz'ya inequality in the local case in \cite{FLHSM}. As in that paper we use an averaging argument \'a la Gagliardo--Nirenberg to reduce the multi-dimensional case to the one-dimensional case. We describe this reduction in Subsection \ref{sec:reduc} and establish the required 1D inequality in Subsection \ref{sec:key}.

\subsection{Reduction to one dimension}\label{sec:reduc}

The key ingredient in our proof of Theorem \ref{thm:HSM} is the following pointwise estimate of a function on an interval.

\begin{lem}\label{lem:dim1}
Let $0<s<1$, $q\geq 1$ and $p\geq 2$ with $ps>1$. Then there is a $c=c(s,q,p)<\infty$ such that for all $f\in C_c^\infty(-1,1)$
\begin{equation}\label{eq:aim}
\|f\|_\infty^{p+q(ps-1)} \leq
c \left( \int_{-1}^1 \int_{-1}^1 \frac{|f(x)-f(y)|^p}{|x-y|^{1+ps}}\,dy\,dx -\mathcal D_{1,p,s} \int_{-1}^1 \frac{|f(x)|^p}{(1-|x|)^{ps}}\,dx
\right) 
\|f\|_q^{q(ps-1)}.
\end{equation}
\end{lem}

Due to the particular form of the exponents this inequality has a scale-invariant form.

\begin{cor}\label{cor:key}
Let $0<s<1$, $q\geq 1$ and $p\geq 2$ with $ps>1$. Then, with the same constant $c=c(p,s,q)<\infty$ as in Lemma~\ref{lem:dim1}, we have for all open sets $\Omega\subsetneq \R$ and all $f\in C_c^\infty(\Omega)$
\begin{equation}\label{eq:aimOmega}
\|f\|_\infty^{p+q(ps-1)} \leq
c \left( \int_\Omega \int_\Omega \frac{|f(x)-f(y)|^p}{|x-y|^{1+ps}}\,dy\,dx -\mathcal D_{1,p,s} \int_\Omega \frac{|f(x)|^p}{d(x)^{ps}}\,dx
\right) 
\|f\|_q^{q(ps-1)}
\end{equation}
where $d(x)=\dist(x,\Omega^c)$.
\end{cor}

\begin{proof}
From Lemma~\ref{lem:dim1}, by translation and dilation, we obtain (\ref{eq:aimOmega}) for any interval and half-line.
The extension to arbitrary open bounded sets is straightforward.
\end{proof}

We prove Lemma \ref{lem:dim1} in Subsection \ref{sec:key}. Now we show how this corollary allows us to deduce our main theorem. Taking advantage of an averaging formula of Loss and Sloane \cite{LossSloane} the argument is almost the same as in \cite{FLHSM}, but we reproduce it here to make this paper self-contained.

\begin{proof}[Proof of Theorem~\ref{thm:HSM}]
Let $\omega_1$, \ldots, $\omega_N$ be an orthonormal basis in $\R^N$. We write $x_j$ for the $j$-th coordinate of $x\in \R^N$ in this basis, and $\tilde{x}_j = x- x_j\omega_j$.
By skipping the $j$-th coordinate of $\tilde{x}_j$ (which is zero), we may regard $\tilde{x}_j$ as an element of $\R^{N-1}$.
 For a~given domain $\Omega\subsetneq\R^N$
we write
\[
 d_j(x) = d_{\omega_j}(x) = \inf\{ |t| : x+t\omega_j \not\in \Omega \}.
\]
If $u\in C_c^\infty(\Omega)$, then Corollary~\ref{cor:key} yields
\[
 |u(x)| \leq C (g_j(\tilde{x}_j) h_j(\tilde{x}_j) )^{\frac{1}{p+q(ps-1)}}
\]
for any $1\leq j \leq N$, where
\begin{align*}
 g_j(\tilde{x}_j) = & \int_{\tilde x_j +a\omega_j\in \Omega} \! da \int_{\tilde x_j +b\omega_j\in\Omega} \! db \ \frac{ |u(\tilde x_j+a\omega_j)-u(\tilde{x}_j+b\omega_j)|^p }{|a-b|^{1+ps}} \\
 & - \mathcal D_{1,p,s} \int_{\R} da\,  \frac{|u(\tilde x_j + a\omega_j)|^p}{d_j(\tilde x_j + a\omega_j)^{ps}}
\end{align*}
and
\begin{align*}
 h_j(\tilde{x}_j) = \left( \int_{\R} da\, |u(\tilde x_j + a\omega_j)|^q \right)^{ps-1} .
\end{align*}
Thus
\[
  |u(x)|^N \leq C^N \prod_{j=1}^N (g_j(\tilde{x}_j) h_j(\tilde{x}_j) )^{\frac{1}{p+q(ps-1)}} \,.
\]
We now pick $q=\frac{pN}{N-ps}$ and rewrite the previous inequality as
\[
 |u(x)|^q \leq C^q \prod_{j=1}^N (g_j(\tilde{x}_j) h_j(\tilde{x}_j) )^{\frac{1}{ps(N-1)}}.
\]
By a standard argument based on repeated use of H\"older's inequality (see, e.g., \cite[Lemma~2.4]{FLHSM}) we deduce that
\[
 \int_{\R^N} |u(x)|^q\,dx \leq C^q \prod_{j=1}^N 
   \left( \int_{\R^{N-1}}  g_j(y)^{\frac{1}{ps}} h_j(y)^{\frac{1}{ps}}\,dy \right)^{\frac{1}{N-1}} \,.
\]
We note that
\[
   \| h_j^{\frac{1}{ps-1}} \|_{L^1(\R^{N-1})} = \| u \|_{L^q(\R^N)}^q \qquad \text{for every } j=1,\ldots, N
\]
and derive from the H\"older and the arithmetic--geometric mean inequality that
\begin{align*}
 \prod_{j=1}^N  \int_{\R^{N-1}}  g_j(y)^{\frac{1}{ps}} h_j(y)^{\frac{1}{ps}}\,dy 
 &\leq
  \prod_{j=1}^N \|g_j\|_1^{\frac{1}{ps}} \|h_j^{\frac{1}{ps-1}} \|_1^{\frac{ps-1}{ps}} 
  =\| u \|_q^{\frac{q(ps-1)N}{ps}} \prod_{j=1}^N \|g_j\|_1^{\frac{1}{ps}}\\
&\leq \| u \|_q^{\frac{q(ps-1)N}{ps}} \left( N^{-1} \sum_{j=1}^N \|g_j\|_1 \right)^{\frac{N}{ps}}.
\end{align*}
To summarize, we have shown that
\[
 \|u\|_q^p \leq C^{\frac{p^2s(N-1)}{N-ps}} N^{-1} \sum_{j=1}^N \|g_j\|_1.
\]
We now average this inequality over all choices of the coordinate system $\omega_j$. We recall the Loss--Sloane formula \cite[Lemma~2.4]{LossSloane}
\begin{align*}
& \int_{\Omega}\int_{\Omega} \frac{|u(x)-u(y)|^p}{|x-y|^{N+ps} } \,dx\,dy \\
& = \frac12 \int_{\Sph^{N-1}} \!\!d\omega \int_{\{x:\, x\cdot\omega=0\}} \!\!d\mathcal L_\omega(x) \int_{x+a\omega\in\Omega} \!\!da
\int_{x+b\omega\in\Omega} \!\!db \ \frac{|u(x+a\omega)-u(x+b\omega)|^p}{|a-b|^{1+ps}},
\end{align*}
where $\mathcal L_\omega$ is $(N-1)$-dimensional Lebesgue measure on the hyperplane $\{x:\, x\cdot\omega=0\}$. Thus we arrive at
\begin{align*}
\|u\|_q^p \leq \frac{2\, C^{\frac{p^2s(N-1)}{N-ps}}}{|\Sph^{N-1}|}  \bigg(&
\int_{\Omega}\int_{\Omega} \frac{|u(x)-u(y)|^p}{|x-y|^{N+ps} } \,dx\,dy\\
& -  \mathcal D_{1,p,s}\frac{\pi^{\frac{N-1}{2}} \Gamma(\frac{1+ps}{2}) }{\Gamma(\frac{N+ps}{2}) }
 \int_\Omega \frac{|u(x)|^p}{m_{ps}(x)^{ps}} \,dx \bigg) \,.
\end{align*}
Recalling the definition of $\mathcal D_{N,p,s}$ we see that this is the inequality claimed in Theorem \ref{thm:HSM}.
\end{proof}


\subsection{Proof of the key inequality}\label{sec:key}

Our first step towards the proof of Proposition \ref{lem:dim1} is a Hardy inequality on an interval with a remainder term. Note the similarity to Lemma \ref{Hball}.

\begin{lem}\label{lem:prem}
 Let $0<s<1$ and $p\geq 2$ with $ps>1$. Then
\begin{align*}
& \int_0^1 \int_0^1 \frac{|f(x)-f(y)|^p}{|x-y|^{1+ps} } \,dx\,dy - \mathcal D_{1,p,s} \int_0^1 \frac{|f(x)|^p}{x^{ps}}\,dx \\
& \quad \geq c_p \int_0^1 \int_0^1 \frac{|v(x)-v(y)|^p}{|x-y|^{1+ps} } \omega(x)^{p/2}\omega(y)^{p/2} \,dx\,dy
+ \int_0^1 W_{p,s}(x) |v(x)|^p \omega(x)^p\,dx
\end{align*}
for all $f$ with $f(0)=0$ (and no boundary condition at $x=1$). Here $\omega(x)=x^{(ps-1)/p}$ and $f=\omega v$. The function $W_{p,s}$ is bounded away from zero and satisfies
\[
W_{p,s}(x) \approx x^{-(p-1)(ps-1)/p} \qquad\text{for}\ x\in(0,1/2]
\]
and
\[
W_{p,s}(x) \approx
\begin{cases}
 1 & \text{if}\ p-1-ps>0 \,, \\
|\ln(1-x)| & \text{if} \ p-1-ps=0 \,, \\
(1-x)^{-1-ps+p}  & \text{if} \ p-1-ps<0 \,,
\end{cases}
\qquad \text{for}\ x\in [1/2,1).
\]
\end{lem}

\begin{proof}
 The general ground state representation \cite{FrSe1} reads
\begin{align*}
\int_0^1 \int_0^1 \frac{|f(x)-f(y)|^p}{|x-y|^{1+ps} } \,dx\,dy
&\geq \int_0^1 V(x) |f(x)|^p \\
&+ c_p \int_0^1 \int_0^1 \frac{|v(x)-v(y)|^p}{|x-y|^{1+ps} } \omega(x)^{p/2}\omega(y)^{p/2} \,dx\,dy
\end{align*}
with
$$
V(x) := 2 \omega(x)^{-p+1} \int_0^1 \left(\omega(x) -\omega(y)\right)
\left|\omega(x) - \omega(y) \right|^{p-2} |x-y|^{-1-ps}\,dy
$$
(understood as principal value integral). We decompose
\begin{align*}
V(x) & = 2 \omega(x)^{-p+1} \int_0^\infty \left(\omega(x) -\omega(y)\right)
\left|\omega(x) - \omega(y) \right|^{p-2} |x-y|^{-1-ps}\,dy \\
& \qquad - 2 \omega(x)^{-p+1} \int_1^\infty \left(\omega(x) -\omega(y)\right)
\left|\omega(x) - \omega(y) \right|^{p-2} |x-y|^{-1-ps}\,dy \\
& = \frac{\mathcal D_{1,p,s}}{x^{ps}} + W_{p,s}(x) \,.
\end{align*}
(The computation of the first term is in 
\cite[Lemma 2.4]{FrankSeiringer}.) For $x\in(0,1)$, the second term is positive, indeed,
$$
W_{p,s}(x) = 2 \omega(x)^{-p+1} \int_1^\infty \left(\omega(y) -\omega(x)\right)^{p-1} (y-x)^{-1-ps}\,dy \,.
$$
Note that at $x=0$
$$
 \int_1^\infty \omega(y)^{p-1} y^{-1-ps}\,dy = c_{p,s} <\infty 
$$
since $ps-(p-1)(ps-1)/p>0$. Hence $W_{p,s}(x) \sim 2c_{p,s} x^{-(p-1)(ps-1)/p}$ as $x\to 0$. On the other hand, at $x=1$, we have
$$
\int_1^\infty \left(\omega(y) -1\right)^{p-1} (y-1)^{-1-ps}\,dy =\tilde c_{p,s} <\infty
\qquad \text{if}\ p-1-ps>0 \,.
$$
Hence $W_{p,s}(x) \to 2\tilde c_{p,s}$ as $x\to 1$ in that case. In the opposite case, one easily finds that for $x=1-\epsilon$, to leading order only $y$'s with $y-1$ of order $\epsilon$ contribute. Hence $W_{p,s}(x) \sim 2\tilde c_{p,s} (1-x)^{-1-ps+p}$ as $x\to 1$ if $p-1-ps<0$ and $W_{p,s}(x) \sim 2\tilde c_{p,s} |\ln(1-x)|$ if $p-1-ps=0$. 
\end{proof}

\begin{cor}\label{cor:prem}
 Let $0<s<1$ and $p\geq 2$ with $ps>1$. Then
\begin{align*}
& \int_{-1}^1 \int_{-1}^1 \frac{|f(x)-f(y)|^p}{|x-y|^{1+ps} } \,dx\,dy - \mathcal D_{1,p,s} \int_{-1}^1 \frac{|f(x)|^p}{(1-|x|)^{ps}}\,dx \\
& \qquad \geq c_p \left( \int_{-1}^0 \int_{-1}^0 + \int_0^1\int_0^1\right)
   \frac{|v(x)-v(y)|^p}{|x-y|^{1+ps} } \omega(x)^{p/2}\omega(y)^{p/2} \,dx\,dy\\
&\qquad\quad+ c_{p,s} \int_{-1}^1 |v(x)|^p \omega(x)\,dx
\end{align*}
for all $f$ with $f(-1)=f(1)=0$. Here $\omega(x)=(1-|x|)^{(ps-1)/p}$ and $f=\omega v$.
\end{cor}

\begin{proof}
The corollary follows by applying Lemma~\ref{lem:prem} to functions $f_1(x)=f(1+x)$ and $f_2(x)=f(1-x)$, where $x\in [0,1]$,
and adding resulting inequalities.
\end{proof}

The second ingredient besides Lemma \ref{lem:prem} in our proof of Proposition \ref{lem:dim1} is the following bound due to Garsia, Rodemich and Rumsey \cite{GRR}.

\begin{lem}
Let $p,s>0$ with $ps>1$. Then for any continuous function $f$ on $[a,b]$
\begin{equation}\label{eq:key}
\int_a^b \int_a^b \frac{|f(x)-f(y)|^p}{|x-y|^{1+ps}}\,dy\,dx \geq c\ \frac{|f(b)-f(a)|^p}{(b-a)^{ps-1}}
\end{equation}
with $c=(ps-1)^p (8(ps+1))^{-p}/4$.
\end{lem}

\begin{proof}
This follows by taking $\Psi(x)=|x|^p$ and $p(x)=|x|^{s+1/p}$ in \cite[Lemma 1.1]{GRR}.
\end{proof}

After these preliminaries we can now turn to the

\begin{proof}[Proof of Proposition \ref{lem:dim1}]
Let $\omega(x)=(1-|x|)^{(ps-1)/p}$. Substituting $v=f/\omega$ and applying Corollary~\ref{cor:prem}, we see that it suffices to prove
\begin{align}\label{eq:aim2}
& \|v\omega\|_\infty^{p+q(ps-1)} \leq
c\bigg( \int_{-1}^1 |v(x)|^p \omega(x)\,dx \\
&\quad+
 \left( \int_{-1}^0 \int_{-1}^0 + \int_0^1\int_0^1\right)
   \frac{|v(x)-v(y)|^p}{|x-y|^{1+ps} } \omega(x)^{p/2}\omega(y)^{p/2} \,dx\,dy
 \bigg)
\|v\omega\|_q^{q(ps-1)}.\nonumber
\end{align}
Without loss of generality, we may assume that $v$ is non-negative and that for some $x_0\in [0,1)$ we have $v(x_0)\omega(x_0)=\|v\omega\|_\infty > 0$. Let $c_1=\omega(\frac{1}{2})/(2\omega(0)) \in (0,1)$.
We distinguish three cases.

\emph{Case 1:} $x_0\in [0, \frac{1}{2}]$
and $v\omega \geq c_1 v(x_0)\omega(x_0)$ on $[0, \frac{1}{2}]$.
Then $\int_{-1}^1 |v|^p\omega \geq 
\int_{0}^{1/2} |v|^p\omega^p \geq \frac{c_1^p}{2} |v(x_0)\omega(x_0)|^p$ and
$\int_{-1}^1 |v\omega|^q \geq \frac{c_1^q}{2} |v(x_0)\omega(x_0)|^q$,
 hence (\ref{eq:aim2}) follows.

\emph{Case 2:} $x_0\in [0, \frac{1}{2}]$ and there is a $z \in [0, \frac{1}{2}]$ such that $v(z)\omega(z) \leq c_1 v(x_0)\omega(x_0)$.
Let $z$ be closest possible to $x_0$, so that $v(z)\omega(z)= c_1 v(x_0)\omega(x_0)$ and
$v\omega \geq c_1 v(x_0)\omega(x_0)$ on the interval $I$ with endpoints $x_0$ and $z$.
We observe that
\[
 v(z) = c_1 v(x_0) \frac{\omega(x_0)}{\omega(z)} = \frac{v(x_0)}{2} \frac{\omega(x_0)}{\omega(0)} \frac{\omega(\frac{1}{2})}{\omega(z)} \leq \frac{v(x_0)}{2}.
\]
We have  by~(\ref{eq:key})
\begin{align*}
  \int_0^1 \int_0^1 & \frac{|v(x)-v(y)|^p}{|x-y|^{1+ps}}\,\omega(x)^{p/2}\omega(y)^{p/2}\,dy\,dx  \\
 & \geq w({\textstyle \frac{1}{2}})^p \int_I \int_I \frac{|v(x)-v(y)|^p}{|x-y|^{1+ps}}\,dy\,dx \\
 & \geq c |v(x_0)-v(z)|^p |z-x_0|^{1-ps} \geq c' |v(x_0)\omega(x_0)|^p |z-x_0|^{1-ps}.
\end{align*}
On the other hand,
\begin{align*}
 \int_{-1}^1 |v\omega|^q &\geq \int_I |v\omega|^q \geq c_1^q |v(x_0)\omega(x_0)|^q |z-x_0| \,.
\end{align*}
Hence (\ref{eq:aim2}) follows.

\emph{Case 3:}  $x_0 \in (\frac{1}{2},1)$.
Since the function $x\mapsto \omega(x)/\omega(\frac{x}{2})$ is decreasing on $[0,1)$, we have that
\[
 \frac{\omega(x_0)}{\omega(x_0/2)} \leq \frac{\omega(1/2)}{\omega(1/4)} =: c_2.
\]
Since $v(\frac{x_0}{2})\omega(\frac{x_0}{2}) \leq v(x_0)\omega(x_0)$, we get that $v(\frac{x_0}{2}) \leq c_2 v(x_0)$. Hence there exists
$z\in [\frac{x_0}{2}, x_0)$ such that $v(z)=c_2v(x_0)$ and $v \geq c_2v(x_0)$ on $[z, x_0]$.
We have by~(\ref{eq:key})
\begin{align*}
  \int_0^1 \int_0^1 &\frac{|v(x)-v(y)|^p}{|x-y|^{1+ps}}\,\omega(x)^{p/2}\omega(y)^{p/2}\,dy\,dx  \\
 & \geq w(x_0)^p \int_{z}^{x_0} \int_{z}^{x_0} \frac{|v(x)-v(y)|^p}{|x-y|^{1+ps}}\,dy\,dx  \\
 & \geq c w(x_0)^p |v(x_0)-v(z)|^p |z-x_0|^{1-ps} \geq c' |v(x_0)\omega(x_0)|^p |z-x_0|^{1-ps}.
\end{align*}
Also,
\[
 \int_{-1}^1 |v\omega|^q \geq \omega(x_0)^q \int_z^{x_0} |v|^q \geq c_2^q |v(x_0)\omega(x_0)|^q |z-x_0| \,.
\]
and again (\ref{eq:aim2}) follows. This completes the proof of Proposition \ref{lem:dim1}.
\end{proof}


\begin{thebibliography}{10}

\bibitem{BFT}
G.~Barbatis, S.~Filippas, and A.~Tertikas.
\newblock A unified approach to improved {$L^p$} {H}ardy inequalities with best
  constants.
\newblock {\em Trans. Amer. Math. Soc.}, 356(6):2169--2196, 2004.

\bibitem{BFL}
R.~D. Benguria, R.~L. Frank, and M.~Loss.
\newblock The sharp constant in the {H}ardy-{S}obolev-{M}az'ya inequality in
  the three dimensional upper half-space.
\newblock {\em Math. Res. Lett.}, 15(4):613--622, 2008.

\bibitem{KBBD-bc}
K.~Bogdan and B.~Dyda.
\newblock The best constant in a fractional {H}ardy inequality.
\newblock {\em Math. Nachr.}, 284(5-6):629--638, 2011.

\bibitem{ChenKumagai}
Z.-Q. Chen and T.~Kumagai.
\newblock Heat kernel estimates for stable-like processes on {$d$}-sets.
\newblock {\em Stochastic Process. Appl.}, 108(1):27--62, 2003.

\bibitem{Davies}
E.~B. Davies.
\newblock {\em Spectral theory and differential operators}, volume~42 of {\em
  Cambridge Studies in Advanced Mathematics}.
\newblock Cambridge University Press, Cambridge, 1995.

\bibitem{DydaHint}
B.~Dyda.
\newblock Fractional {H}ardy inequality with a remainder term.
\newblock {\em Colloq. Math.}, 122(1):59--67, 2011.

\bibitem{FLHSM}
R.~L. Frank and M.~Loss.
\newblock {H}ardy-{S}obolev-{M}az'ya inequalities for arbitrary domains.
\newblock {\em J. Math. Pures Appl.}, to appear, 2011.

\bibitem{FrSe1}
R.~L. Frank and R.~Seiringer.
\newblock Non-linear ground state representations and sharp {H}ardy
  inequalities.
\newblock {\em J. Funct. Anal.}, 255(12):3407--3430, 2008.

\bibitem{FrankSeiringer}
R.~L. Frank and R.~Seiringer.
\newblock Sharp fractional {H}ardy inequalities in half-spaces.
\newblock In {\em Around the research of {V}ladimir {M}az'ya. {I}}, volume~11
  of {\em Int. Math. Ser. (N. Y.)}, pages 161--167. Springer, New York, 2010.

\bibitem{GRR}
A.~M. Garsia, E.~Rodemich, and H.~Rumsey, Jr.
\newblock A real variable lemma and the continuity of paths of some {G}aussian
  processes.
\newblock {\em Indiana Univ. Math. J.}, 20:565--578, 1970/1971.

\bibitem{MR2214908}
Q.-Y. Guan and Z.-M. Ma.
\newblock Reflected symmetric {$\alpha$}-stable processes and regional
  fractional {L}aplacian.
\newblock {\em Probab. Theory Related Fields}, 134(4):649--694, 2006.

\bibitem{LossSloane}
M.~Loss and C.~Sloane.
\newblock Hardy inequalities for fractional integrals on general domains.
\newblock {\em J. Funct. Anal.}, 259(6):1369--1379, 2010.

\bibitem{M}
V.~G. Maz'ja.
\newblock {\em Sobolev spaces}.
\newblock Springer Series in Soviet Mathematics. Springer-Verlag, Berlin, 1985.
\newblock Translated from the Russian by T. O. Shaposhnikova.

\bibitem{Sloane}
C.~A. Sloane.
\newblock A fractional {H}ardy-{S}obolev-{M}az'ya inequality on the upper halfspace.
\newblock {\em Proc. Amer. Math. Soc.}, 139:4003--4016, 2011.

\end{thebibliography}

\def\cprime{$'$}

\end{document}